\documentclass[a4paper,12pt]{amsart}
\usepackage{amssymb}
\usepackage{amsmath}
\usepackage{color}

\newcommand{\hide}[1]{}
\newcommand{\ignore}[1]{}

\newcommand{\C}{\mathbb{C}}

\newcommand{\Z}{\mathbb{Z}}

\newcommand{\F}{\mathbb{F}}

\newcommand{\Legendre}[2]{\genfrac{(}{)}{}{}{#1}{#2}}

\newtheorem{dummy}{Dummy}

\newtheorem{theorem}[dummy]{Theorem}
\newtheorem*{theorem*}{Theorem}

\theoremstyle{definition}

\theoremstyle{remark}
\newtheorem{rem}[dummy]{Remark}
\newtheorem*{rem*}{Remark}
\begin{document}
\bibliographystyle{amsalpha}

\author{S.~Mattarei}
\email{mattarei@science.unitn.it}
\urladdr{http://www-math.science.unitn.it/\~{ }mattarei/}
\address{Dipartimento di Matematica\\
  Universit\`a degli Studi di Trento\\
  via Sommarive 14\\
  I-38050 Povo (Trento)\\
  Italy}

\title[Partial sums of central binomial coefficients and Catalan numbers]{Asymptotics of partial sums of central binomial coefficients and Catalan numbers}

\begin{abstract}
We compute higher order asymptotic expansions for the partial sums of the sequences of central binomial coefficients and Catalan numbers,
$\sum_{k=0}^n\binom{2k}{k}$
and
$\sum_{k=0}^nC_n$.
We also obtain closed forms for the polynomials
$\sum_{k=0}^{q-1}\binom{2k}{k}x^k$
and
$\sum_{k=0}^{q-1}C_kx^k$
over the field of $p$ elements, where $q$ is a power of the prime $p$.
\end{abstract}


\maketitle

\thispagestyle{empty}

\section{Introduction}

For $n$ a natural number, the $n$th {\em Catalan number} is defined by
\[
C_n=\frac{1}{n+1}\binom{2n}{n}=\binom{2n}{n}-\binom{2n}{n+1}.
\]
These numbers have a wealth of combinatorial interpretations and applications;
several dozens of them are collected in~\cite[Corollary~6.2.3, Exercises~6.19--6.36]{Stanley:EC2}.
As the definition shows, the Catalan numbers are closely related to the {\em central binomial coefficients}
$\binom{2n}{n}$.
Partial sums of these sequences ({\sf A000108} and {\sf A006134} in~\cite{Sloane:OEIS}), and related sums, have recently attracted some interest,
especially with regard to the values of some of them modulo powers of a prime~\cite{PanSun,SunTau:Catalan,SunTau:central_binomial}.
It may be worth noting that, because of the identity $\binom{2k}{k}=\sum_{j=0}^k\binom{k}{j}^2$,
the sum $\sum_{k=0}^n\binom{2k}{k}$ equals the square of the Frobenius norm
of the lower triangular matrix made of the first $n+1$ rows of Pascal's triangle.

In an expanded preprint version of~\cite{SunTau:Catalan}
({\sf arXiv:0709.1665v5}, 23 September 2007),
the authors included some conjectures on
$\sum_{k=0}^n\binom{2k}{k}$
and related sums.
One of them, Conjecture~5.2, concerns the asymptotic behaviour of two such sums;
it claims that
\begin{equation}\label{eq:conj}
\sum_{k=0}^n\binom{2k}{k}
\sim
\frac{4^{n+1}}{3\sqrt{\pi n}}
\qquad\text{and}\qquad
\sum_{k=0}^nC_k
\sim
\frac{4^{n+1}}{3n\sqrt{\pi n}}.
\end{equation}
The standard notation ``$\sim$'' indicates that the ratio of the two sides tends to one as $n$ tends to infinity.
In the first part of this note we prove more precise asymptotics than those in Equation~\eqref{eq:conj},
by means of generating functions and standard methods of asymptotic analysis.
We also discuss the variations
$\sum_{k=0}^n\binom{2k}{k}\alpha^k$
and $\sum_{k=0}^nC_k\alpha^k$
of our partial sums, where $\alpha$ is a complex constant.

In the second part of the paper we consider the values of those partial sums modulo a fixed prime $p$, that is to say,
in a field of characteristic $p$.
Limits and asymptotic expansions are generally meaningless in this discrete setting,
but there is a sense in which an evaluation of an infinite series, such as
\[
\sum_{k=0}^\infty\binom{2k}{k}\alpha^k=\frac{1}{\sqrt{1-4\alpha}},
\quad\text{for $\alpha\in\C$ with $|\alpha|<1/4$},
\]
has corresponding finite sums over a range $0\le k<q$ as its prime characteristic analogues, where $q$ is a power of $p$.
We illustrate this point of view by giving very simple proofs, based on generating functions,
of some known congruences of this type.
They are special cases of more general congruences which were obtained in~\cite{PanSun,SunTau:Catalan,SunTau:central_binomial}
by means of elementary but substantially more intricate arguments.
By contrast, we only use simple power series manipulations
and the congruence $(1+x)^{q}\equiv 1+x^q\pmod{p}$.

\section{Asymptotics}

Equation~\eqref{eq:conj} is not hard to prove using
the fact that
$\binom{2n}{n}\sim 4^n/\sqrt{\pi n}$,
which follows from Stirling's approximation $n!\sim n^ne^{-n}\sqrt{2\pi n}$.
Indeed, for positive sequences the conditions $a_n\sim b_n$
and $\sum_{k=1}^nb_k\to\infty$,
imply
$\sum_{k=1}^na_k\sim\sum_{k=1}^nb_k$.
Consequently, we have
$\sum_{k=0}^n\binom{2k}{k}
\sim
\sum_{k=0}^n 4^{k}/\sqrt{\pi k}$.
The asymptotic behaviour of the latter is easily found.
For, on the one hand we have
$\sum_{k=0}^n 4^k/\sqrt{k}
\ge (1/\sqrt{n})\sum_{k=0}^n 4^k
=(4^{n+1}-1)/(3\sqrt{n})$.
On the other hand, for any $1\le m\le n$ we have
\[
\sum_{k=0}^n\frac{4^k}{\sqrt{k}}
\le
\frac{1}{\sqrt{m}}\sum_{k=m}^n 4^k
+\sum_{k=0}^{m-1}4^k
\le
\frac{4^{n+1}}{3\sqrt{m}}+\frac{4^m}{3},
\]
and taking
$m=\lfloor n-\log_4 n\rfloor$
we find that
\[
\sum_{k=0}^n\frac{4^k}{\sqrt{k}}
\le
\frac{4^{n+1}}{3\sqrt{n-\log_4 n}}+\frac{4^{n+1}}{3n}.
\]
Taken together, these inequalities imply the former estimate in Equation~\eqref{eq:conj}, and the latter can be shown similarly.
Now we proceed to computing more precise asymptotic estimates using the classical method of Darboux.

\begin{theorem}\label{thm:central}
We have
\[
\sum_{k=0}^n\binom{2k}{k}
=
\frac{4^{n+1}}{3\sqrt{\pi n}}
\left(1+\frac{1}{24\,n}+\frac{59}{384\,n^2}+\frac{2425}{9216\,n^3}+O(n^{-4})\right).
\]
\end{theorem}

\begin{proof}
Because
$\binom{2n}{n}=(-4)^{n}\binom{-1/2}{n}$,
the central binomial coefficients have the generating function
\begin{equation}\label{eq:central_gf}
\sum_{n=0}^\infty\binom{2n}{n}x^n
=
\sum_{n=0}^\infty(-4)^n\binom{-1/2}{n}x^n
=
\frac{1}{\sqrt{1-4x}},
\end{equation}
see~\cite[Equation~(2.5.11)]{Wilf}.
It follows that
\[
\sum_{n=0}^\infty\biggl(\sum_{k=0}^n\binom{2k}{k}\biggr)x^n
=
\biggl(\sum_{k=0}^\infty\binom{2k}{k}x^k\biggr)\biggl(\sum_{h=0}^\infty x^{h}\biggr)
=
\frac{1}{\sqrt{1-4x}}\cdot\frac{1}{1-x},
\]
which we conveniently write in the equivalent form
\begin{equation}\label{eq:genfunc}
\sum_{n=0}^\infty 4^{-n}\biggl(\sum_{k=0}^n\binom{2k}{k}\biggr)z^n
=
\frac{1}{\sqrt{1-z}}\cdot\frac{4}{4-z},
\end{equation}
having set $z=4x$.
This function is analytic on the unit disk and has only one singularity on its boundary, at $z=1$,
which is of algebraic type.
Because
\[
\frac{4}{4-z}
=
\frac{4}{3}\cdot\frac{1}{1+\bigl((1-z)/3\bigr)}
=
\frac{4}{3}\sum_{j=0}^\infty
(-3)^{-j}(1-z)^j,
\]
the Puiseux expansion at $z=1$ of the function in Equation~\eqref{eq:genfunc} reads
\begin{equation*}%\label{eq:expansion}
\frac{1}{\sqrt{1-z}}\cdot\frac{4}{4-z}
=
\frac{4}{3}\sum_{j=0}^\infty
(-3)^{-j}(1-z)^{j-1/2}.
\end{equation*}
According to Darboux's lemma (see~\cite[Theorem~5.3.1]{Wilf},
\cite[Theorem~8.4]{Szego:orthogonal}
or~\cite[p.~277]{Com}),
the asymptotic expansion for the coefficient of $z^n$ in the generating function
is formally obtained by adding up the coefficients of $z^n$
in each term of its Puiseux expansion at $z=1$.
Truncating and adding the appropriate $O$ term we find that
\begin{equation}\label{eq:Darboux}
4^{-n}\sum_{k=0}^n\binom{2k}{k}
=
\frac{4}{3}\sum_{j=0}^m
(-3)^{-j}\binom{n-j-1/2}{n}
+O(n^{-m-3/2}),
\end{equation}
for any nonnegative integer $m$.
Here we have transformed the binomial coefficients involved using the standard formula
$\binom{-a}{n}=(-1)^n\binom{n+a-1}{n}$.

The former estimate in Equation~\eqref{eq:conj}
follows by setting $m=0$ in Equation~\eqref{eq:Darboux},
using $4^n\binom{n-1/2}{n}=\binom{2n}{n}=(2n)!/(n!)^2$
and Stirling's formula.
Of course, higher values of $m$ provide sharper results, as we now exemplify.
Note that
$4^n\binom{n-1/2}{n}=\binom{2n}{n}$,
whence
\begin{align*}
4^n\binom{n-j-1/2}{n}
&=
\binom{2n}{n}
\cdot
\frac{(-1/2)(-3/2)\cdots\bigl((-2j+1)/2\bigr)}{(n-1/2)(n-3/2)\cdots(n-j+1/2)}.
\\&=
\binom{2n}{n}
\frac{(-1)^j}{(2n-1)(2n/3-1)\cdots\bigl(2n/(2j-1)-1\bigr)}
\end{align*}
Plugging this into Equation~\eqref{eq:Darboux} multiplied by $4^n$ we obtain
\begin{equation}\label{eq:central}
\begin{aligned}
\sum_{k=0}^n\binom{2k}{k}
&=
\frac{4}{3}\binom{2n}{n}
\sum_{j=0}^m
\frac{3^{-j}}{(2n-1)(2n/3-1)\cdots\bigl(2n/(2j-1)-1\bigr)}
\\&\quad+O(4^nn^{-m-3/2}),
\end{aligned}
\end{equation}
where the summand is interpreted to take the value $1$ when $j=0$.
For example, when $m=3$ one may use
\[
\binom{2n}{n}=\frac{4^n}{\sqrt{\pi n}}
\left(1-\frac{1}{8n}+\frac{1}{128\,n^2}+\frac{5}{1024\,n^3}+O(n^{-4})\right)
\]
(see~\cite[Exercise~9.60]{GKP}, for example)
and obtain the conclusion as given in the statement of the theorem after lengthy calculations.
It is certainly best to let a symbolic manipulation system such as {\sc Maple} do the calculations, starting directly from the right-hand side of Equation~\eqref{eq:central}.
\end{proof}

\begin{theorem}\label{thm:Catalan}
We have
\[
\sum_{k=0}^nC_n
=
\frac{4^{n+1}}{3n\sqrt{\pi n}}
\left(1-\frac{5}{8\,n}+\frac{475}{384\,n^2}+\frac{1225}{9216\,n^3}+O(n^{-4})\right).
\]
\end{theorem}

\begin{proof}
The generating function for the Catalan numbers,
\begin{equation}\label{eq:Catalan_gf}
\sum_{n=0}^\infty C_nx^n
=
\frac{1-\sqrt{1-4x}}{2x},
\end{equation}
can be obtained by integrating Equation~\eqref{eq:central_gf} and adjusting the constant term, see~\cite[Equation~(2.5.10)]{Wilf}.
Proceeding in a similar fashion as in the proof of Theorem~\ref{thm:central}, we have
\[
\sum_{n=0}^\infty\biggl(\sum_{k=0}^n C_k\biggr)x^n
=
\frac{1-\sqrt{1-4x}}{2x(1-x)},
\]
which we rewrite in the equivalent form
\[
\sum_{n=0}^\infty 4^{-n}\biggl(\sum_{k=0}^n C_k\biggr)z^n
=
8\frac{1-\sqrt{1-z}}{z(4-z)},
\]
where $z=4x$.
Now
\[
\frac{8}{z(4-z)}=\frac{2}{z}+\frac{2}{4-z}
=
2\sum_{j=0}^{\infty}(1-z)^j
+\frac{2}{3}
\sum_{j=0}^{\infty}(-3)^{-j}(1-z)^j,
\]
and hence the Puiseux expansion at $z=1$ of our generating function is
\[
8\frac{1-\sqrt{1-z}}{z(4-z)}
=
\frac{2}{3}
\sum_{j=0}^{\infty}\bigl(3-(-3)^{-j}\bigr)\bigl((1-z)^j-(1-z)^{j+1/2}\bigr).
\]
Darboux's lemma tells us that
\begin{equation*}
\begin{aligned}
4^{-n}\sum_{k=0}^n C_k
&=
-\frac{2}{3}\sum_{j=0}^m
\bigl(3-(-3)^{-j}\bigr)\binom{n-j-3/2}{n}
+O(n^{-m-5/2}),
\end{aligned}
\end{equation*}
for any nonnegative integer $m$.
Note that the coefficient $\binom{j}{n}$ of $z^n$ in $(1-z)^j$ in the Puiseux expansion
gives no contribution to this estimate, because it vanishes as soon as $n>m$;
put differently, those terms of the Puiseux expansion add up to a part of the generating function which is analytic at $1$.
Similar calculations as in the previous case lead to
\begin{equation*}%\label{eq:Catalan}
\begin{aligned}
\sum_{k=0}^n C_k
&=
\frac{2}{3}\binom{2n}{n}
\sum_{j=0}^m
\frac{3\cdot(-1)^{j}+3^{-j}}{(2n-1)(2n/3-1)\cdots\bigl(2n/(2j+1)-1\bigr)}
\\&\quad+O(4^nn^{-m-5/2}).
\end{aligned}
\end{equation*}
A {\sc Maple} calculation with $m=4$ returns the desired estimate.
\end{proof}

As concluding remarks, note that this classical method applies similarly to produce asymptotics for the modified sums
$\sum_{k=0}^n \alpha^k\binom{2k}{k}$
and
$\sum_{k=0}^n \alpha^kC_k$,
where $\alpha$ is a complex constant.
We refrain from computing higher order asymptotics for these more general sums
and limit ourselves to finding the first order asymptotic expansion for the former sum (the other being similar).
A case distinction is necessary according to the value of $\alpha$.

When $|\alpha|>1/4$, the generating function in
\begin{equation}\label{eq:genfunc_modif}
\sum_{n=0}^\infty (4\alpha)^{-n}\biggl(\sum_{k=0}^n\alpha^k\binom{2k}{k}\biggr)z^n
=
\frac{1}{\sqrt{1-z}}\cdot\frac{4\alpha}{4\alpha-z},
\end{equation}
which generalizes Equation~\eqref{eq:genfunc},
still has $z=1$ as its dominant singularity (that is, the one of smallest modulus).
The case $m=0$ of the corresponding modification of Equation~\eqref{eq:Darboux} then yields
\[
\sum_{k=0}^n\alpha^k\binom{2k}{k}
=
\frac{(4\alpha)^{n+1}}{4\alpha-1}
\binom{n-1/2}{n}
\bigl(1+O(n^{-3/2})\bigr)
=
\frac{(4\alpha)^{n+1}}{(4\alpha-1)\sqrt{\pi n}}
\bigl(1+O(n^{-3/2})\bigr).
\]

When $|\alpha|<1/4$,
the pole at $z=4\alpha$ becomes the dominant singularity of the generating function in Equation~\eqref{eq:genfunc_modif}.
After bringing this singularity to the unit circle by the
substitution $z=4\alpha x$, an application of Darboux's method with $m=0$ yields
\[
\sum_{k=0}^n\alpha^k\binom{2k}{k}
=
\frac{1}{\sqrt{1-4\alpha}}
\bigl(1+O(n^{-1})\bigr),
\]
where the square root in the formula refers to the branch of $\sqrt{1-z}$ (on the domain $|z|<1$) which evaluates to $1$ when $z=0$.

When $|\alpha|=1/4$ and $\alpha\neq 1/4$, the generating function in Equation~\eqref{eq:genfunc_modif}
has two singularities of modulus one,
and the method of Darboux requires adding up the contributions from both singularities.
If we are content with a first order asymptotic expansion, we have
\[
\sum_{k=0}^n\alpha^k\binom{2k}{k}
=
\frac{1}{\sqrt{1-4\alpha}}
\bigl(1+O(n^{-1/2})\bigr),
\]
where the leading term comes from the pole at $z=1$, and the other singularity contributes to the error term.

Finally, when $\alpha=1/4$ the function in Equation~\eqref{eq:genfunc_modif} is simply $(1-z)^{-3/2}$.
In this case our sum admits a closed form
\[
\sum_{k=0}^{n}4^{-k}\binom{2k}{k}
=
\sum_{k=0}^{n}(-1)^{k}\binom{-1/2}{k}
=
(-1)^n\binom{-3/2}{n}
=\frac{2n+1}{4^n}\binom{2n}{n}
\]
(see sequence {\sf A002457} in~\cite{Sloane:OEIS}),
and hence
\[
\sum_{k=0}^{n-1}4^{-k}\binom{2k}{k}
=2\sqrt{n/\pi}
\left(1+\frac{3}{8n}-\frac{7}{128\,n^2}+\frac{9}{1024\,n^3}+O(n^{-4})\right).
\]

\section{Partial sums modulo a prime}

Let $q$ be a power of a prime $p$.
We will show that
\begin{equation}\label{eq:sums_mod_p}
\sum_{k=0}^{q-1}\binom{2k}{k}\equiv
\Legendre{q}{3}\pmod{p},
\quad\text{and}\quad
\sum_{k=0}^{q-1}C_k\equiv
\frac{3\Legendre{q}{3}-1}{2}\pmod{p},
\end{equation}
where $\Legendre{a}{3}$ is a Legendre symbol, and hence is uniquely determined in this case by
$\Legendre{a}{3}\in\{0,\pm 1\}$ and
$\Legendre{a}{3}\equiv a\pmod{p}$.
We will do that by finding closed expressions for the polynomials
$\sum_{k=0}^{q-1}\binom{2k}{k}x^k$ and $\sum_{k=0}^{q-1}C_kx^k$ over the field of $p$ elements,
after which the substitution $x=1$ will give the desired conclusion.
It will be convenient to assume that $p$ is odd.
However, when $p=2$ congruences~\eqref{eq:sums_mod_p} follow from the power series congruences
$\sum_{k=0}^{\infty}\binom{2k}{k}x^k\equiv 1\pmod{2}$ and
$\sum_{k=0}^{\infty}C_kx^k\equiv \sum_{i=0}^\infty x^{2^i-1}\pmod{2}$,
which are easy to prove directly, starting from the definitions of binomial coefficients
and Catalan numbers.

\begin{theorem}\label{thm:pol_central}
If $q$ is a power of an odd prime $p$ we have
\begin{equation*}%\label{eq:gf_cbc_q}
\sum_{k=0}^{q-1}\binom{2k}{k}x^k\equiv(1-4x)^{(q-1)/2}\pmod{p}.
\end{equation*}
\end{theorem}

\begin{proof}
Recall from Equation~\eqref{eq:central_gf} that
$
\sum_{k=0}^\infty\binom{2k}{k}x^k
=
(1-4x)^{-1/2},
$
and work in the formal power series ring $\Z[[x]]$.
Basic facts about binomial coefficients and Fermat's Little Theorem imply that $(1-4x)^q\equiv 1-(4x)^q\equiv 1-4x^q\pmod{p}$.
Therefore, noting that all binomial power series involved have integral coefficients, we have
\begin{equation*}%\label{eq:1-4x}
\begin{aligned}
(1-4x)^{-1/2}
&=
(1-4x)^{(q-1)/2}\bigl((1-4x)^q\bigr)^{-1/2}
\\&\equiv
(1-4x)^{(q-1)/2}(1-4x^q)^{-1/2}
\pmod{p}
\\&\equiv
(1-4x)^{(q-1)/2}
\pmod{(x^q,p)}.
\end{aligned}
\end{equation*}
This last congruence, with respect to the modulus $(x^q,p)$, means that the polynomial, obtained from the power series at the left-hand side
by discarding all terms of degree $q$ or higher, is congruent modulo $p$ to the polynomial at the right-hand side.
The desired conclusion now follows.
\end{proof}

\begin{theorem}\label{thm:pol_Catalan}
If $q$ is a power of an odd prime $p$ we have
\begin{equation*}
\sum_{k=0}^{q-1}C_kx^k\equiv\frac{1-(1-4x)^{(q+1)/2}}{2x}-x^{q-1}\pmod{p}.
\end{equation*}
\end{theorem}

\begin{proof}
In this case we need to carry the calculation of the previous proof one step further, and obtain
\[
(1-4x)^{-1/2}
\equiv
(1-4x)^{(q-1)/2}
+2x^q
\pmod{(x^{q+1},p)},
\]
whence
\[
(1-4x)^{1/2}
=
(1-4x)^{-1/2}(1-4x)
\equiv
(1-4x)^{(q+1)/2}
+2x^q
\pmod{(x^{q+1},p)}.
\]
Therefore, we have
\[
\sum_{n=0}^\infty C_nx^n
=
\frac{1-\sqrt{1-4x}}{2x}
\equiv
\frac{1-(1-4x)^{(q+1)/2}}{2x}-x^{q-1}\pmod{(x^q,p)},
\]
which implies the desired conclusion.
\end{proof}

Congruences~\eqref{eq:sums_mod_p}, for odd $p$, follow by evaluating on $x=1$ the polynomials of Theorems~\ref{thm:pol_central} and~\ref{thm:pol_Catalan},
using the fact that
$(-3)^{(q-1)/2}
\equiv\bigl(\frac{q}{3}\bigr)
\pmod{p}$.
This is easy to show
 %In fact, $(-3)^{(q-1)/2}\equiv 0,1,-1\pmod{p}$
 %according as whether $-3$ is zero, a square or a nonsquare in $\F_q$,
 %and then there are various ways of seeing that
 %$(-3)^{(q-1)/2}
 %\equiv\bigl(\frac{q}{3}\bigr)
 %\pmod{p}$.
either by using Jacobi symbols and Gauss' quadratic reciprocity law,
or by viewing $-3$ as the discriminant of the polynomial
$(x^3-1)/(x-1)=x^2+x+1$, and then evaluating on $-3$ the quadratic character of the finite field $\F_q$.

\begin{rem}
It appears that congruences~\eqref{eq:sums_mod_p} were first established in~\cite[Theorem~1.2 and Corollary~1.3]{PanSun}.
Those results have the restriction $q=p$, but have $\binom{2k}{k+d}$ and $C_{k+d}$ in place of $\binom{2k}{k}$ and $C_k$,
where $d$ is an integer with $0\le d<q$.
It may be possible to extend our method in order to cover this variation.

The results of~\cite{PanSun} include congruences for similar sums where $\binom{2k}{k}$ and $C_k$
are multiplied by fixed powers of $k$, or divided by $k$.
The results of the former type can easily be recovered from Theorems~\ref{thm:pol_central} and~\ref{thm:pol_Catalan}
by repeated application of the operator $xD$, that is, differentiation followed by multiplication by $x$.
As an example, from Theorem~\ref{thm:pol_central} we have
\begin{equation*}
\sum_{k=1}^{q-1}k\binom{2k}{k}x^k
\equiv
xD(1-4x)^{(q-1)/2}
\equiv
2x(1-4x)^{(q-3)/2}
\pmod{p}
\end{equation*}
for $p$ odd. Assuming $p>3$ for simplicity and substituting $x=1$ we obtain
\begin{equation*}
\sum_{k=0}^{q-1}k\binom{2k}{k}
\equiv
2(-3)^{(q-3)/2}
\equiv
\frac{-2}{3}\Legendre{q}{3}
\pmod{p}
\end{equation*}
When $q=p$ this is the case $d=0$ of the second assertion of~\cite[Theorem~1.2]{PanSun}.
\end{rem}

\begin{rem}
According to~\cite[Corollary~1.1]{SunTau:Catalan},
congruences~\eqref{eq:sums_mod_p} actually hold modulo $p^2$, with a change for $p=3$ in case of the former sum.
We do not know analogues modulo $p^2$ of our Theorems~\ref{thm:pol_central} and~\ref{thm:pol_Catalan}.
However, our approach based on those results has the advantage of readily allowing an evaluation of the modified sums
$\sum_{k=0}^{q-1}\binom{2k}{k}\alpha^k$ and $\sum_{k=0}^{q-1}C_k\alpha^k$ modulo $p$,
where $\alpha$ is any algebraic integer.
In the special case where $\alpha$ is an ordinary integer prime to $p$,
the former evaluation can also be obtained by solving a linear recurrence given in~\cite[Theorem~1.1]{SunTau:central_binomial}.
\end{rem}

\bibliography{References}

\end{document}